\documentclass[11pt,twoside]{article}
\usepackage{authblk}
\usepackage[scaled=0.8]{DejaVuSansMono}
\usepackage{emptypage}
\usepackage[a4paper,marginratio={5:5,5:7},width=147mm]{geometry}
\usepackage[T2A,T1]{fontenc}
\usepackage[utf8x]{inputenc}

\usepackage{fancyhdr}
\newcommand{\maybebf}{}

\fancypagestyle{plain}{%
    \fancyhf{}%
    \fancyfoot[C]{\maybebf \thepage}
    
}
\fancypagestyle{intro}{
    \renewcommand{\sectionmark}[1]{\markboth{##1}{}}
    \fancyhf{}
    \fancyhead[LE,RO]{\leftmark}
    \fancyfoot[C]{\maybebf \thepage}

}

\renewcommand{\sectionmark}[1]{%
  \markboth{%
    \ifnum\value{section}>0
      \maybebf{\thesection.}\space
    \fi
    #1%
  }{%
    \ifnum\value{subsection}=0
            \thesection. #1
        \fi%
  }%
}
\renewcommand{\subsectionmark}[1]{%
    \markright{%
        \ifnum\value{subsection}>0
            \thesubsection. #1
        \fi%
    }%
}
\usepackage[english]{babel}
\usepackage{graphicx}
\usepackage{amsmath}
\usepackage{amssymb}
\usepackage{amsthm}
\usepackage{mathtools}
\usepackage{microtype}
\usepackage{tikz}
\usetikzlibrary{arrows.meta,arrows,decorations.pathmorphing,shapes,knots,hobby,quotes,positioning}
\usepackage{asymptote}
\usepackage{tikz-cd}
\usepackage{parskip}
\usepackage{float}
\usepackage{caption}
\usepackage{subcaption}
\usepackage{enumitem}
\usepackage{adjustbox}
\usepackage{xfakebold}
\usepackage{mdframed}
\usepackage{booktabs}
\usepackage{multirow}
\usepackage{siunitx}
\usepackage{nicematrix}
\NiceMatrixOptions{xdots/shorten=0.87em, renew-dots}

\usepackage[calc,showdow,english]{datetime2}
\DTMnewdatestyle{mydateformat}{%
}

\usepackage{comment}
\usepackage{import}
\usepackage{standalone}
\usepackage{xcolor}
\usepackage{csquotes}
\usepackage[doi=false,isbn=false,
maxbibnames=99,backend=biber,style=numeric,sorting=nyt]{biblatex}
\usepackage[anythingbreaks]{breakurl}
\usepackage{hyperref}
\definecolor{hlinkcol}{HTML}{A00000}
\definecolor{hcitecol}{HTML}{308030}
\hypersetup{
    colorlinks,
    linkcolor={teal!100!black},
    citecolor={green!50!black},
    urlcolor={blue!0!black}
}

\theoremstyle{plain}
\newtheorem{thm}{Theorem}[section]
\newtheorem*{thm:main30}{Theorem \ref*{thm:main30}}

\newtheorem*{thm:embed}{Theorem \ref*{thm:embed20}}
\newtheorem{prop}[thm]{Proposition}
\newtheorem{lemma}[thm]{Lemma}
\newtheorem{cor}[thm]{Corollary}

\newtheorem*{thm*}{Theorem}
\newtheorem*{lemma*}{Lemma}
\newtheorem*{prop*}{Proposition}
\newtheorem*{conj*}{Conjecture}
\newtheorem*{cor*}{Corollary}

\theoremstyle{definition}
\newtheorem{defin}[thm]{Definition}

\newtheorem{question}[thm]{Question}
\newtheorem*{defin*}{Definition}

\theoremstyle{remark}
\newtheorem{remark}[thm]{Remark}

\counterwithin{figure}{section}
\counterwithin{table}{section}

\newcommand{\id}{\mathrm{id}}
\newcommand{\ZZ}{\mathbb Z}
\newcommand{\QQ}{\mathbb Q}
\newcommand{\RR}{\mathbb R}

\newcommand{\HH}{\mathbb H}

\newcommand{\til}{\widetilde}

\tikzset{
dot/.style = {circle, fill, minimum size=#1,
              inner sep=0pt, outer sep=0pt},
dot/.default = 5pt
}

\DeclarePairedDelimiter\abs{\lvert}{\rvert}%
\DeclarePairedDelimiter\norm{\lVert}{\rVert}%

\makeatletter
\let\oldabs\abs
\def\abs{\@ifstar{\oldabs}{\oldabs*}}
\let\oldnorm\norm
\def\norm{\@ifstar{\oldnorm}{\oldnorm*}}
\makeatother

\DeclareMathOperator{\image}{im}

\DeclareMathOperator{\Hom}{Hom}

\DeclareMathOperator{\diag}{diag}
\DeclareMathOperator{\vspan}{span}
\DeclareMathOperator{\Fac}{Fac}
\newcommand{\Sq}{\mathrm{Sq}}

\newcommand{\fbseries}{
\unskip\setBold\aftergroup\unsetBold\aftergroup\ignorespaces}
\makeatletter
\newcommand{\setBoldness}[1]{\def\fake@bold{#1}}
\makeatother

\tikzset{
    rots/.style={anchor=south, rotate=90, inner sep=.5mm}
}

\newcommand{\spinc}{\ensuremath{\text{spin}^c}}
\newcommand{\pinc}{\ensuremath{\text{pin}^c}}

\newcommand\restr[2]{{%
  \left.\kern-\nulldelimiterspace %
  #1 %
  \vphantom{\big|} %
  \right|_{#2} %
  }}

\author{Jacopo G. Chen\thanks{Scuola Normale Superiore, Pisa, Italy. Email: \href{mailto:jacopo.chen@sns.it}{\texttt{jacopo.chen@sns.it}}}
}
\title{Closed hyperbolic manifolds without \spinc{} structures}

\begin{document}

\DTMsetdatestyle{mydateformat}
\date{}
\maketitle

\begingroup
\centering\small
\textbf{Abstract}\par\smallskip
\begin{minipage}{\dimexpr\paperwidth-9.5cm}
In all dimensions $n \ge 5$, we prove the existence of closed orientable hyperbolic manifolds that do not admit any \spinc{} structure, and in fact we show that there are infinitely many commensurability classes of such manifolds. These manifolds all have non-vanishing third Stiefel--Whitney class $w_3$ and are all arithmetic of simplest type. More generally, we show that for each $k \ge 1$ and $n \ge 4k+1$, there exist infinitely many commensurability classes of closed orientable hyperbolic $n$-manifolds $M$ with $w_{4k-1}(M) \ne 0$.
\end{minipage}
\par\endgroup
\section{Introduction}\label{sec:intro}
Let $M$ be a smooth $n$-manifold. Various properties of the tangent bundle $TM$, such as orientability and existence of spin and \spinc{} structures, 
have a relatively simple characterization involving the Stiefel--Whitney classes $w_i(M) \coloneqq w_i(TM)$. Indeed, as we discuss in Section~\ref{sec:sw}, we have:
\begin{itemize}
    \item \makebox[2.7cm][l]{$M$ is orientable} $\iff$ $w_1(M) = 0$;
    \item \makebox[2.7cm][l]{$M$ is spin} $\iff$ $w_1(M) = w_2(M) = 0$;
    \item \makebox[2.7cm][l]{$M$ is \spinc{}} $\iff$ $w_1(M) = 0$ and $w_2(M)$ lifts to $H^2(M; \ZZ)$.
\end{itemize}
These can be taken as definitions; moreover, the lifting condition for $w_2(M)$ is equivalent to the vanishing of the \emph{third integral Stiefel--Whitney class} $W_3(M) \in H^3(M; \ZZ)$, which reduces modulo $2$ to $w_3(M)$ for orientable manifolds $M$.

A spin manifold is also orientable and \spinc{}. We also recall that, in dimension up to $3$, every closed orientable manifold is spin; moreover, by a theorem of Hirzebruch and Hopf~\cite{hirzebruch-hopf}, every closed orientable $4$-manifold is \spinc{} (see also
~\cite{teichner-vogt} for non-closed $4$-manifolds).

Some recent research has been focused on finding orientable hyperbolic manifolds that do not admit spin or \spinc{} structures.
In~\cite{virtual-spin}, Long and Reid found cusped non-spin manifolds in all dimensions $\ge 5$, while Martelli, Riolo and Slavich~\cite{non-spin} found compact non-spin manifolds in all dimensions $\ge 4$.
Even more recently, Reid and Sell~\cite{non-spinc-cusped} showed that there are infinitely many commensurability classes of non-\spinc{} cusped manifolds in all dimensions $\ge 6$.

The main result of this paper covers the compact non-\spinc{} case:
\begin{thm}\label{thm:main20}
    For every $n \ge 5$, there exist infinitely many commensurability classes of closed orientable hyperbolic $n$-manifolds that have no \spinc{} structure.
\end{thm}
Since having no \spinc{} structure also implies having no spin structure, this provides an alternate proof to the main result of~\cite{non-spin} in dimension $> 4$.

In similar fashion to~\cite{non-spin}, the proof technique relies on codimension-$1$ geodesic embeddings of arithmetic manifolds. More specifically, using results of Kolpakov, Reid and Slavich~\cite{embedding}, we extend~\cite[Lemma~5.1]{non-spin} to the following theorem, which may be of independent interest:

\begin{thm}\label{thm:embed20}
    Let $M^n$ be a $k$-good hyperbolic $n$-manifold and let $c \in H^1(M^n; \ZZ_2)$. Then $M^n$
    geodesically embeds in a $k$-good hyperbolic $(n + 1)$-manifold $M^{n+1}$ such that \begin{equation}
        w_1(\nu_{M^{n+1}}(M^n)) = c.
    \end{equation}
    If $c = w_1(M^n)$, we can take $M^{n+1}$ to be orientable.
\end{thm}
(The definition of \emph{$k$-good}, for $k$ a totally real number field, is given in Section~\ref{sec:main-thm}.)
We use this result to recursively construct an infinite sequence of manifolds
\begin{equation}
    (M^1, M^2, M^3, M^4, M^5, \dots),
\end{equation}
starting from a $k$-good circle $M^1$. By carefully choosing the class $c$ at each step and applying the Whitney sum formula, we can ensure the following chain of implications:
\begin{equation}
\begin{split}
    w_0(M^1) \ne 0
    &\implies w_1(M^2) \ne 0 \implies w_2(M^3) \ne 0 \\
    &\implies w_3(M^4) \ne 0 \implies w_3(M^i) \ne 0\ \text{for $i \ge 5$}.
\end{split}
\end{equation}
Moreover, we also have that $M^i$ is orientable if $i \ge 5$. We then use the fact that different fields $k$ yield non-commensurable manifolds to obtain infinitely many commensurability classes of non-\spinc{} manifolds in each dimension, proving Theorem~\ref{thm:main20}.

In fact, the same argument can be used to prove a generalization of Theorem~\ref{thm:main20}:
\begin{thm}\label{thm:main30}
    For every $m \ge 1$ and $n \ge 4m+1$, there exist infinitely many commensurability classes of closed orientable hyperbolic $n$-manifolds $M$ such that $w_{4m-1}(M) \ne 0$.
\end{thm}
A direct consequence is that there exist closed orientable hyperbolic manifolds with nonzero Stiefel--Whitney classes in arbitrarily high degree.

Since the proof of Theorem~\ref{thm:main20} is non-constructive, in Section~\ref{sec:explicit} we also give a semi-explicit procedure to obtain a compact $5$-manifold without \spinc{} structures, based on two right-angled hyperbolic polytopes (the dodecahedron and the $120$-cell) and on a Coxeter $5$-polytope discovered by Bugaenko~\cite{bugaenko}.

Some natural questions arise. For instance, regarding the existence of more non-\spinc{} manifolds, we could try to escape the constraints given by our arithmetic tools:
\begin{question}
    Do there exist closed, orientable, non-arithmetic hyperbolic manifolds without \spinc{} structures?
\end{question}
\begin{question}
    Are there infinitely many commensurability classes of such manifolds?
\end{question}
A variation of our method based on the results of~\cite{embedding-na} might give an affirmative answer to both questions.

Furthermore, as the vanishing of $W_3$ does not \emph{a priori} imply that of $w_3$, we could ask:
\begin{question}
    Do there exist closed orientable hyperbolic manifolds without \spinc{} structures and with $w_3 = 0$?
\end{question}
Using characteristic classes, one could also investigate {higher spin structures} beyond spin and \spinc{}: for instance, the primary obstruction to a $\text{spin}^h$ structure is the fifth integral Stiefel--Whitney class $W_5$~\cite{spinh}. While this class is outside the scope of Theorem~\ref{thm:main30}, it is still natural to ask:
\begin{question}
    What can be said about the existence of orientable hyperbolic manifolds without higher spin structures?
\end{question}
Finally, we describe a potential alternate method for the proof of Theorems~\ref{thm:main20} and~\ref{thm:main30}.
In another paper~\cite{noncobordant}, the author shows the existence of hyperbolic $n$-manifolds that do not bound any $(n+1)$-manifold for infinitely many $n$, including $n=5$; the latter case corresponds to the non-vanishing of the Stiefel--Whitney number $w_2 w_3$, which implies $w_3 \ne 0$. In general, the technique also relies on the Kolpakov--Reid--Slavich embedding, to construct $n$-manifolds for which certain sums of Stiefel--Whitney numbers do not vanish, and may be adapted to obtain the non-vanishing of specific Stiefel--Whitney classes along with orientability; compare~\cite[Question~1.5]{noncobordant} and the preceding discussion.
\subsection*{Structure of the paper}
In Section~\ref{sec:sw}, we recall some properties of the Stiefel--Whitney classes, and how they relate to the existence of \spinc{} structures. Then, in Section~\ref{sec:main-thm}, we prove Theorems~\ref{thm:main20} and~\ref{thm:main30} using arithmetic methods. Finally, we give an alternate construction of a non-\spinc{} hyperbolic $5$-manifold in Section~\ref{sec:explicit}.
\subsection*{Acknowledgments}
I am grateful to my advisor Bruno Martelli for bringing my attention to this problem and for his guidance during the writing of this paper. I would also like to thank Alexander Kolpakov, Alan W.\ Reid and Leone Slavich for a discussion on Theorem~\ref{thm:embed20}.
\section{Stiefel\texorpdfstring{--}{-}Whitney classes}\label{sec:sw}
In this section, we reframe the problem of finding non-\spinc{} manifolds in terms of characteristic classes.

Let $E \twoheadrightarrow M$ be a real vector bundle over a manifold $M$. We denote the \emph{Stiefel--Whitney classes} of $E$ by $w_i(E) \in H^i(M; \ZZ_2)$ for $i\ge 0$, where the class $w_0(E)$ is defined to be $1\in H^0(M; \ZZ_2)$. These classes can be collected into an element of the cohomology ring
\begin{equation}
    w(E) \coloneqq 1 + w_1(E) + w_2(E) + \dots \in H^*(M; \ZZ_2),
\end{equation}
called the \emph{total Stiefel--Whitney class} of $E$. We also have the \emph{Whitney sum formula}
\begin{equation}
    w(E \oplus E') = w(E) w(E')
\end{equation}
for two bundles $E, E'$ over $M$.
When $E$ is the tangent bundle $TM$, we speak of the Stiefel--Whitney classes of $M$ and write $w_i(M), w(M)$.

The following result will be very useful in the proof of the main theorem.

\begin{lemma}\label{lemma:whitney-codim1}
    Let $j\colon N\hookrightarrow M$ be a codimension-$1$ embedding of manifolds. Then, for all $i \ge 1$, we have \begin{equation}
        j^*(w_i(M)) = w_i(N) + w_{i-1}(N)w_1(\nu_M(N)).
    \end{equation}
\end{lemma}
\begin{proof}
    We have $j^*(TM) = \restr{TM}{N} \simeq TN\oplus \nu_M(N)$. The statement follows from naturality of $w_i$ and the Whitney sum formula.
\end{proof}

We also note the existence of some algebraic relations, arising from the \emph{Wu formulas}, that hold for all closed manifolds of a given dimension; an example in dimension $4$, that we will use in the following, is $w_1w_2 = 0$. More generally, we have:

\begin{prop}\label{prop:wu}
    Let $M$ be a closed $4k$-manifold. Then we have $w_1(M)w_{4k-2}(M) = 0$.
\end{prop}
\begin{proof}
    Following~\cite{milnor-stasheff}, we can express the total Stiefel--Whitney class as the total Steenrod square of the total Wu class: $w(M) = \Sq(v(M))$. In particular, writing simply $w_i, v_i$ for $w_i(M), v_i(M)$, we have:
    \begin{align}
        w_1 w_{4k-2} &= \Sq^1(w_{4k-2}) + w_{4k-1} \\
        &= \Sq^1(w_{4k-2}) + \sum_{i = 0}^{4k-1} \Sq^i(v_{4k-1-i}) \\
        &= \Sq^1(w_{4k-2}) + \sum_{i = 0}^{2k-1} \Sq^{2i+1}(v_{4k-2-2i})
        + \sum_{i = 0}^{2k-1} \Sq^{2i}(v_{4k-1-2i}) \\
        &= \Sq^1\bigg(w_{4k-2} + \sum_{i = 0}^{2k-1} \Sq^{2i}(v_{4k-2-2i})\bigg)
        + \sum_{i = 0}^{2k-1} \Sq^{2i}(v_{4k-1-2i}) \\
        &= \Sq^1\bigg(\sum_{i = 0}^{2k-2} \Sq^{2i+1}(v_{4k-3-2i})\bigg)
        + \sum_{i = 0}^{2k-1} \Sq^{2i}(v_{4k-1-2i}) \\
        &= \sum_{i = 0}^{2k-1} \Sq^{2i}(v_{4k-1-2i}), \label{eq:sum-wu}
    \end{align}
    using the relations $\Sq^{2i+1} = \Sq^1\Sq^{2i}$ and $\Sq^1\Sq^{2i+1} = 0$. Finally, the sum~(\ref{eq:sum-wu}) is zero since the Wu classes of a closed manifold vanish above the middle dimension $2k$, and $\Sq^i(v_j) = 0$ for $i > j$.
\end{proof}

Stiefel--Whitney classes provide obstructions to various structures on a manifold $M$. In particular:
\pagebreak[2]
\begin{itemize}
    \item \makebox[2.7cm][l]{$M$ is orientable} $\iff$ $w_1(M) = 0$;
    \item \makebox[2.7cm][l]{$M$ is spin} $\iff$ $w_1(M) = w_2(M) = 0$;
    \item \makebox[2.7cm][l]{$M$ is \pinc{}} $\iff$ $w_2(M)$ lifts to the integral cohomology group $H^2(M; \ZZ)$;
    \item \makebox[2.7cm][l]{$M$ is \spinc{}} $\iff$ $w_1(M) = 0$ and $M$ is \pinc{}.
\end{itemize}
The conditions for having \pinc{} or \spinc{} structures are arguably in a less usable form than the first two, especially for finding counterexamples; in order to resolve this issue, we introduce the \emph{integral Stiefel--Whitney classes}.

The short exact sequence of groups
\begin{equation}
    0 \longrightarrow \ZZ \xrightarrow{\ {\cdot 2}\ } \ZZ \longrightarrow \ZZ_2 \longrightarrow 0
\end{equation}
induces a long exact sequence on cohomology
\begin{equation}
    \dots \longrightarrow H^i(M; \ZZ) \longrightarrow H^i(M; \ZZ)\xrightarrow{\ \rho\ } H^i(M; \ZZ_2)\xrightarrow{\ \beta\ } H^{i+1}(M; \ZZ)\longrightarrow \dots,
\end{equation}
where $\rho$ is the reduction modulo $2$ and $\beta$ is the \emph{Bockstein homomorphism}. We can now define the $(i+1)$-th \emph{integral Stiefel--Whitney class}:
\begin{equation}
    W_{i+1}(M) \coloneqq \beta(w_i(M)) \in H^{i+1}(M; \ZZ).
\end{equation}
By exactness, we immediately see that:
\begin{itemize}
    \item \makebox[2cm][l]{$M$ is \pinc{}} $\iff$ $W_3(M) = 0$;
    \item \makebox[2cm][l]{$M$ is \spinc{}} $\iff$ $W_3(M) = 0$ and $w_1(M) = 0$.
\end{itemize}
\begin{remark}\label{rmk:w3}
    It is well known that the reduction of $W_3$ modulo $2$ is the {first Steenrod square} of the second Stiefel--Whitney class $\Sq^1(w_2) = w_2 w_1 + w_3$, which is much easier to compute than its integral counterpart. Indeed, we have $\Sq^1(w_2) = w_3$ for closed orientable manifolds (or closed $4$-manifolds, by the Wu formulas).
\end{remark}

As such, our general strategy will be to construct orientable manifolds with non-vanishing third Stiefel--Whitney class.

\section{The main theorem}\label{sec:main-thm}
The aim of this section is to prove Theorem~\ref{thm:main20}. The proof relies on a technique similar to that of \cite[Section~5]{non-spin}: we recursively construct a sequence of manifolds $(M^i)_{i\ge 1}$ by embedding each one as a totally geodesic submanifold of the next. We start by extending \cite[Lemma~5.1]{non-spin} to possibly non-orientable manifolds.
    
\begin{defin}
    Let $k$ be a totally real number field, that is, a finite extension of $\QQ$ whose complex embeddings have image in $\RR$. Assume $k \ne \QQ$. We say that a hyperbolic manifold (or orbifold) is \emph{$k$-good} if it is closed and arithmetic of simplest type, with admissible quadratic form $f$ defined over $k$ and with fundamental group contained in the $k$-points $\mathrm{O}(f, k)$. (Note that the condition $k \ne \QQ$ implies that the manifold is compact.)
\end{defin}

\begin{lemma}[compare {\cite{embedding}}, {\cite[Lemma~5.1]{non-spin}}]\label{lemma:embed}
    Let $M^n$ be a $k$-good hyperbolic $n$-manifold. Then $M^n$
    geodesically embeds in an orientable $k$-good hyperbolic $(n + 1)$-manifold.
\end{lemma}
\begin{proof}[Sketch of proof]
    Let $f$ be the admissible quadratic form associated to $M^n$. Define a new form $g = f + y^2$ with an additional variable $y$; this form is also admissible. As outlined in \cite{embedding}, we can embed the group $\mathrm{O}(f, k)$ into $\mathrm{O}(g, k)$. By~\cite[Corollary~5.2]{embedding}, this can be improved to an embedding $\mathrm{O}^+(f, k) < \mathrm{SO}^+(g, k)$. A separability argument gives a torsion-free, cocompact subgroup $\Lambda < \mathrm{SO}^+(g, k) < \operatorname{Isom}^+(\HH^{n+1})$ containing $\pi_1(M^n)$, and such that $M^n$ embeds geodesically in the orientable manifold $M^{n+1}\coloneqq \HH^{n+1} / \Lambda$.
\end{proof}
Note that, by orientability of $M^{n+1}$, we have $w_1(\nu_{M^{n+1}}(M^n)) = w_1(M^n)$. 
In fact, Theorem~\ref{thm:embed20} (restated below) generalizes this by realizing any choice of a normal bundle for $M^n$:
\begin{thm:embed}\label{thm:embed2}
    Let $M^n$ be a $k$-good hyperbolic $n$-manifold and let $c \in H^1(M; \ZZ_2)$. Then $M^n$
    geodesically embeds in a $k$-good hyperbolic $(n + 1)$-manifold $M^{n+1}$ such that \begin{equation}
        w_1(\nu_{M^{n+1}}(M^n)) = c.
    \end{equation}
    If $c = w_1(M^n)$, we can take $M^{n+1}$ to be orientable.
\end{thm:embed}
\begin{proof}
    The key part in the proof of Lemma~\ref{lemma:embed} is the embedding $\psi\colon \mathrm{O}^+(f,k) \hookrightarrow \mathrm{SO}^+(g,k)$. This is accomplished in~\cite[Corollary~5.2]{embedding}, on the level of matrices, by sending
    \begin{equation}
        \renewcommand{\arraystretch}{1.5}
        P \mapsto \left[
        \begin{array}{c|c}
            \det(P) & 0 \\ \hline
            0 & P
        \end{array}\right].
    \end{equation}
    Let us instead define an embedding $\psi_c\colon \pi_1(M) \hookrightarrow \mathrm{O}^+(g,k)$ by
    \begin{equation}
        \renewcommand{\arraystretch}{1.5}
        P \mapsto \left[
        \begin{array}{c|c}
            c(P) & 0 \\ \hline
            0 & P
        \end{array}\right],
    \end{equation}
    where $c$ is seen as a map $\pi_1(M) \to \ZZ_2 \simeq \{\pm 1\}$. After carrying out the rest of the construction in the same way\footnote{Here an essential fact is that $\image(\psi_c) < \operatorname{Isom}(\HH^{n+1})$ is geometrically finite: see the proof of Theorem~1.1 in~\cite[Section 7]{embedding}.}, we obtain a manifold $M^{n+1}$ in which $M^n$ embeds geodesically. Moreover, the normal bundle $E \coloneqq \nu_{M^{n+1}}(M^n)$ is the quotient of $\nu_{\HH^{n+1}}(\HH^n)$ by the action of $\image(\psi_c)$ on its base space. It is now clear that $w_1(E)$, that is, the monodromy of the associated double cover, is simply $c$.

    If $c = w_1(M^n)$, then $\psi_c = \psi$ on $\pi_1(M)$, and the original proof of Lemma~\ref{lemma:embed} ensures global orientability.
\end{proof}
\begin{cor}\label{cor:nonzero-whitney}
    Let $M^n$ be a $k$-good hyperbolic $n$-manifold such that $w_{n-1}(M^n) \ne 0$. Then $M^n$ geodesically embeds in a $k$-good hyperbolic $(n+1)$-manifold $M^{n+1}$ with $w_n(M^{n+1}) \ne 0$.
\end{cor}
\begin{proof}
    By Poincaré duality, we can find $c \in H^1(M^n; \ZZ_2)$ such that \begin{equation}
    w_n(M^n) + w_{n-1}(M^n)c \ne 0.    
    \end{equation}
    We then apply Theorem~\ref{thm:embed2} to this class, obtaining an embedding $j\colon M^n \hookrightarrow M^{n+1}$. By Lemma~\ref{lemma:whitney-codim1}, we have 
    $j^*(w_n(M^{n+1})) = w_n(M^n) + w_{n-1}(M^n)c \ne 0$,
    so $w_n(M^{n+1}) \ne 0$.
\end{proof}

Using this result, we can attack the problem of finding non-\spinc{} manifolds by starting from dimension $1$.

\begin{thm}\label{thm:main2}
    Let $M^1$ be a $k$-good hyperbolic $1$-manifold. Then, for all $n\ge 5$, $M^1$ geodesically embeds in a $k$-good orientable, non-\spinc{} hyperbolic $n$-manifold $M^n$.
\end{thm}
\begin{proof}
    As anticipated, using Theorem~\ref{thm:embed2}, we recursively construct a sequence of $k$-good manifolds $(M^1, M^2, M^3, M^4, M^5, \dots)$, each one embedding geodesically in the next, and such that $M^n$ is orientable and non-\spinc{} for $n \ge 5$. By Remark~\ref{rmk:w3}, the latter condition holds when $w_3(M^n) \ne 0$.

    First, apply Corollary~\ref{cor:nonzero-whitney} three times starting from $M^1$, obtaining $k$-good manifolds $(M^1, M^2, M^3, M^4)$ such that $w_{i-1}(M^i) \ne 0$ (in particular $w_3(M^4) \ne 0$).

    Next, we repeatedly apply Theorem~\ref{thm:embed2} to obtain orientable manifolds $(M^i)_{i\ge 5}$, by choosing the class $c = w_1(M^i)$ at each step. By Lemma~\ref{lemma:whitney-codim1}, we have:
    \begin{align}
        j_i^*(w_3(M^{i+1})) &= w_3(M^i) + w_2(M^i)w_1(M^i)\\
        &= w_3(M^i) && \hspace{-5em}\text{$(*)$}\\
        &\ne 0. && \hspace{-5em}\text{(by induction)}
    \end{align}
    The equality labeled with $(*)$ can be shown to hold by cases: if $i = 4$, then $M^4$ satisfies $w_2(M^4) w_1(M^4) = 0$ by the Wu formulas; if $i \ge 5$, then $w_1(M^i) = 0$ by orientability of $M^i$. It then follows that $w_3(M^{i+1}) \ne 0$, completing the induction.
\end{proof}
We now show the existence of many such $1$-manifolds (i.e. circles).

\begin{lemma}
    For every totally real number field $k\ne \QQ$, there exists a $k$-good hyperbolic $1$-manifold $C_k$.
\end{lemma}
\begin{proof}
    Write $k = \QQ(\alpha)$, with $\alpha$ the largest root of its minimal polynomial, and let $R$ be the ring of integers of $k$. We can find $q \in \QQ$ such that $\delta \coloneqq \alpha - q$ is positive, and all Galois conjugates of $\delta$ are negative. Then, the quadratic form $f\coloneqq \diag(-\delta, 1)$ is admissible.

    Let $\Gamma$ be the arithmetic group of simplest type $\mathrm{O}^+(f, R)$, up to the appropriate conjugation into $\mathrm{O}^+(1,1)$. Then, $\Gamma$ has finite covolume by the Borel--Harish-Chandra theorem, and there exists a torsion-free subgroup $\Gamma_1 < \Gamma$ of finite index, corresponding to a $k$-good $1$-manifold $C_k$.
\end{proof}
\begin{remark}
    The $k$-goodness condition on the circle $C_k$, as an arithmetic $1$-manifold, is but a condition on its length $\operatorname{len}(C_k)$; among other things, it implies $\cosh(\operatorname{len}(C_k))\in k$.
\end{remark}
Since the field of definition $k$ is a commensurability invariant among arithmetic manifolds of simplest type, 
we have, after applying Theorem~\ref{thm:main2} to $C_k$:
\begin{cor*}[{Theorem~\ref{thm:main20}}]
    For every $n \ge 5$, there exist infinitely many commensurability classes of closed orientable hyperbolic $n$-manifolds that have no \spinc{} structure.
\end{cor*}
Incidentally, from Remark~\ref{rmk:w3} and the proof of Theorem~\ref{thm:main2} we also obtain:
\begin{cor}
    There exist infinitely many commensurability classes of closed hyperbolic $4$-manifolds that have no \pinc{} structure.
\end{cor}
Such manifolds must be non-orientable by the Hirzebruch--Hopf theorem~\cite{hirzebruch-hopf}.

Finally, using Proposition~\ref{prop:wu}, we can easily prove Theorem~\ref{thm:main30}, as a generalization of Theorem~\ref{thm:main2}:
\begin{thm:main30}
    For every $m \ge 1$ and $n \ge 4m+1$, there exist infinitely many commensurability classes of closed orientable hyperbolic $n$-manifolds $M$ such that $w_{4m-1}(M) \ne 0$.
\end{thm:main30}
\begin{proof}
    The proof is completely analogous to that of Theorem~\ref{thm:main2}: first, starting from a $k$-good circle $M^1$, we apply Corollary~\ref{cor:nonzero-whitney} to construct a sequence of $k$-good manifolds $(M^1, M^2, \dots, M^{4m})$ such that $w_{i-1}(M^i) \ne 0$. Then, we repeatedly apply Theorem~\ref{thm:embed2} with $c = w_1(M^i)$, obtaining orientable manifolds $(M^i)_{i \ge 4m+1}$. Crucially, we need the Wu formula $w_{4m-2}(M^{4m})w_1(M^{4m}) = 0$ from Proposition~\ref{prop:wu} in order to have
    \begin{align}
        j_{4m}^*(w_{4m-1}(M^{4m+1})) &= w_{4m-1}(M^{4m+1}) + w_{4m-2}(M^{4m+1})w_1(M^{4m+1}) \\ &= w_{4m-1}(M^{4m+1}) \ne 0.
    \end{align}
    The induction proceeds as in Theorem~\ref{thm:main2}, showing that all subsequent manifolds also have $w_{4m-1}(M^{i}) \ne 0$.
\end{proof}
\begin{cor}
    There exist closed orientable hyperbolic manifolds with nonzero Stiefel--Whitney classes in arbitrarily high degree.
\end{cor}
\section{Explicit examples}\label{sec:explicit}
In this section, we outline a semi-explicit procedure to construct closed orientable hyperbolic $5$-manifolds with no \spinc{} structure. This construction is compatible with that of Theorem~\ref{thm:main2} and yields $\QQ(\sqrt{5})$-good manifolds.

\subsection{Right-angled polytopes}
Let us recall some general definitions about polytopes in hyperbolic space.

A \emph{hyperbolic polytope} $P$ is the intersection of finitely many half-spaces in $\HH^n$. The \emph{dimension} of $P$, $\dim P$, is the least dimension of a hyperbolic subspace containing $P$, or $-1$ if $P$ is empty. Such a subspace is called a \emph{supporting subspace} for $P$.

The boundary of $P$ (relative to its supporting subspace) is naturally stratified by dimension and consists of polytopes of dimension $k \in \{0, 1, \dots, \dim P - 1\}$, called the \emph{$k$-faces} of $P$. In particular, faces of dimension $0$ and codimension $2$ and $1$ are respectively called \emph{vertices}, \emph{ridges} and \emph{facets} of $P$.
When two facets meet at a ridge, their \emph{dihedral angle} is well defined, and it is less than $\pi$ by convexity of $P$.

A polytope is said to be \emph{right-angled} if all its dihedral angles are $\pi/2$. Right-angled polytopes are a special case of \emph{Coxeter polytopes}, which have dihedral angles of the form $\pi/k$ for $k \ge 2$. Any face of a right-angled polytope is itself a right-angled polytope.

In what follows, we implicitly assume that every right-angled polytope is \emph{compact}, a condition equivalent to having finite volume and no points at infinity.

\subsection{The coloring method}
The importance of right-angled polytopes stems from the fact that they can be used to construct hyperbolic manifolds with the \emph{coloring method}.

Let $P$ be a right-angled polytope with facets $\Fac(P) = \{F_1, \dots, F_k\}$. By abuse of notation, we may refer to the facet $F_i$ using just the number $i$, and consequently write $\Fac(P) = \{1, \dots, k\}$.

\begin{defin}
    A \emph{(proper) coloring} of $P$ is a map $\lambda\colon \Fac(P) \to V$ for some finite-dimensional vector space $V$ over $\ZZ_2$, such that, when facets $i_1, \dots, i_s$ meet at a vertex, the vectors $\{\lambda(i_1), \dots, \lambda(i_s)\}$ form a linearly independent set. The vector $\lambda(i)$ is called the \emph{color} of facet $i$.
\end{defin}

Given a coloring $\lambda$, we can construct a closed hyperbolic manifold $M(P, \lambda)$. Take $|V|$ copies $\{P_v\}$ of $P$, indexed by vectors $v \in V$. Then, for each facet $F \in \Fac(P)$ and each vector $v \in V$, glue $P_v$ and $P_{v + \lambda(F)}$ along $F$, using the identity map of $F$. 

The resulting space is indeed a closed manifold, sometimes called a \emph{real toric manifold} over $P$. Moreover, the construction is clearly invariant under vector space isomorphisms, so we may assume $V = \ZZ_2^m$ for some $m$, if needed. 
We note that $M(P, \lambda)$ is connected if and only if $\image(\lambda)$ spans $V$; otherwise each connected component arises from gluing some set $\{P_v : v \in C\}$, where $C$ is a coset of $\vspan \image(\lambda)$ in $V$. 

\begin{remark}\label{rmk:hyp-mfd-right}
    This construction can also be carried out when $P$ has nontrivial topology, that is, when it is a manifold with right-angled corners, provided that every facet is embedded, or equivalently, not adjacent to itself along a ridge. This is essentially because such objects are locally modeled on right-angled polytopes and also have a natural stratification.
\end{remark}


\begin{remark}\label{rmk:coloring-extend}
    A \emph{partial} coloring on a proper subset $A \subset \operatorname{Fac}(P)$, with values in a $\ZZ_2$-vector space $V$, can always be extended to a genuine coloring $\operatorname{Fac}(P) \to V \oplus V'$ by coloring the remaining facets with different basis vectors of $V'$.
\end{remark}



It is convenient to identify $P$ with one of its copies inside $M$, say $P_0$. Now, any face $F$ of $P$ is contained in a totally geodesic, possibly disconnected submanifold $M_F$ of the same dimension, obtained as the preimage of $F$ under the covering map $M\twoheadrightarrow P$. The following result provides a description of $M_F$ as a real toric manifold, generalizing~\cite[Lemma~2.8]{induced-coloring}.
\begin{prop}
\label{prop:coloring-sub}
    Let $P$ be a right-angled hyperbolic manifold with embedded facets, and let $F$ be a face of $P$. Let $\operatorname{adj}(F)$ be the set of facets of $P$ that intersect the boundary of $F$ but not its interior, and let $\operatorname{supp}(F)$ be the set of facets of $P$ containing $F$ in their boundary. Then $\Fac(F)$ maps naturally onto $\operatorname{adj}(F)$, and $M_F = M(F, \lambda_F)$, where $\lambda_F$ is the composition
    \begin{equation}
        \Fac(F) \twoheadrightarrow \operatorname{adj}(F) \hookrightarrow \Fac(P) \xrightarrow{\,\,\,\lambda\,\,\,} V \twoheadrightarrow V / \langle \lambda(i) \mid i \in \operatorname{supp}(F)\rangle.
    \end{equation}
\end{prop}
\begin{proof}
    The space $V$ acts by isometries on the $P$-tessellation of $M(P, \lambda)$. This action restricts to the tessellation of $M_F$ by copies of $F$. Let us fix a copy of $F$, say $F_0$; then its stabilizer is precisely $\langle \lambda(i) \mid i \in \operatorname{supp}(F)\rangle$, and the manifold $M_F$ is obtained by gluing copies of $F$, canonically indexed by the quotient $V / \langle \lambda(i) \mid i \in \operatorname{supp}(F)\rangle$. It is not hard to check that the gluing is induced by the coloring $\lambda_F$.
\end{proof}
\begin{defin}
    We say that a cohomology class in $H^1(M(P, \lambda); \ZZ_2)$ is a \emph{sum of hypersurfaces} if it is dual to a sum of codimension-$1$ cycles of the form $M_F$, where $F\in \Fac(P)$. 
\end{defin}

\subsection{Examples of right-angled polytopes}

It can be shown that compact right-angled polytopes do not exist in dimension $n\ge 5$~\cite{right-dim5}, but examples are abundant in lower dimensions. Among these, two are of particular interest to us: the right-angled dodecahedron in dimension $3$, and the right-angled $120$-cell in dimension $4$. These highly symmetric polytopes are hyperbolic analogs of their Euclidean regular counterparts and, as hyperbolic orbifolds, they are both $\QQ(\sqrt{5})$-good.

A simple way to construct either polytope starts by placing its Euclidean counterpart inside the Beltrami-Klein model of $\HH^3$ or $\HH^4$, centered at the origin. This gives a (regular) hyperbolic polytope, since the Beltrami-Klein model preserves collinearity. If we apply a Euclidean scaling to the polytope, taking care that it remains contained within the interior of the disk, we note that the dihedral angle is a continuous function of the scaling factor, and can take on an interval of values. Calculations show that this interval contains $\pi/2$ for both the dodecahedron and the $120$-cell, resulting in the corresponding right-angled polytopes.

\subsection{Small covers}

It is not hard to see that, in a compact right-angled polytope $P$, exactly $\dim P$ facets meet at a vertex. This fact, combined with the linear independence condition on a coloring $\lambda$, implies that the colors belong to a vector space of rank $m \ge \dim P$. The equality case corresponds to \emph{small covers}:

\begin{defin}
    A \emph{small cover} of a compact right-angled $n$-polytope $P$ is a connected real toric manifold $M(P, \lambda)$ with $\lambda \colon \{1,\dots,k\} \to \ZZ_2^n$. 
\end{defin}

These manifolds are especially well studied because they are easier to classify by computer, due to their simplicity, and they have good homological behavior.

Indeed, the small covers of the dodecahedron have been classified by Garrison and Scott~\cite{garrison-scott}, and they form 25 isomorphism classes. More recently, Ma and Zheng~\cite{ma-zheng} performed a partial classification of small covers of the $120$-cell, by considering the case $\lvert{\image(\lambda)}\rvert \le 8$.

As for their algebraic properties, there is an explicit formula for the $\ZZ_2$-cohomology ring of a small cover~\cite[Theorem~4.14]{dj-small-covers}:
\begin{equation}
    H^*(M(P, \lambda); \ZZ_2) = \ZZ_2[a_1, \dots, a_k] / (I + J),
\end{equation}
where the ideal $I$ captures the combinatorics of $P$:
\begin{equation}
    I = \left(a_{i_1}\dots a_{i_s} \mid F_{i_1} \cap \dots \cap F_{i_s} = \varnothing \right),
\end{equation}
and the ideal $J$ captures information about the coloring:
\begin{equation}
    J = \bigg(\sum_{i=1}^k \lambda(i)_j\cdot a_i \bigm| j = 1,\dots,n \bigg).
\end{equation}
The ring is naturally graded by degree, and the $a_i$ are generators for the first cohomology group. Each $a_i$ can be interpreted geometrically by taking the Poincaré dual; the resulting codimension-$1$ homology class is represented by the hypersurface $M_i$ containing facet $i$. As a consequence:
\begin{cor}
    If $M\coloneqq M(P, \lambda)$ is a small cover, then every class $c \in H^1(M, \ZZ_2)$ is a sum of hypersurfaces.
\end{cor}
Higher-degree classes involving products of distinct generators can also be dually interpreted, using the fact that cup products are dual to intersections.
In this view, the ideal $I$ shows that non-intersecting sets of facets give vanishing cup products; as for $J$, its $j$-th generator is dual to the boundary separating the copies $P_v$ with $v_j = 0$ from those with $v_j = 1$.

Finally, the total Stiefel--Whitney class of a small cover has a simple expression in terms of the ring generators~\cite[Corollary~6.8]{dj-small-covers}:
\begin{equation}\label{eq:sw-total}
    w(M(P,\lambda)) = \prod_{i = 1}^k (1 + a_i),
\end{equation}
which can be expanded, in each degree, to obtain
\begin{equation}\label{eq:sw-degree}
    w_d(M(P, \lambda)) = \sum_{\substack{S \subseteq \{1, \dots,k\} \\ |S| = d}} \prod_{i \in S} a_i.
\end{equation}

\subsection{The construction}
Recall the notion of right-angled hyperbolic manifold from Remark~\ref{rmk:hyp-mfd-right}. The following result serves as a constructive substitute of Theorem~\ref{thm:embed2}:

\begin{prop}\label{prop:3in4}
    Let $X$ be a real toric manifold over a right-angled hyperbolic $n$-polytope $P$, with characteristic function $\lambda\colon \{1,\dots,k\}\to \ZZ_2^m$, and let $c \in H^1(X; \ZZ_2)$ be a sum of hypersurfaces. Let $Q$ be a right-angled hyperbolic $(n+1)$-manifold with embedded facets. Assume that $Q$ has a facet isometric to $P$, and that the natural map $\Fac(P) \twoheadrightarrow \operatorname{adj}(P)$ is a bijection (these conditions hold if $Q$ is a polytope). Then $X$ embeds as a totally geodesic submanifold of a real toric manifold $Y$ over $Q$, such that $w_1(\nu_Y(X)) = c$. Moreover, if $c = w_1(X)$, then we can take $Y$ to be orientable.
\end{prop}
\begin{proof}
    For the case $c = w_1(X)$, we can prove that $X$ embeds geodesically in an orientable real toric manifold $Y$ over $Q$, as in~\cite[Proposition~2.9]{induced-coloring}. Hence, $w_1(\nu_Y(X)) = w_1(X) = c$.

    In the general case, as in the previous section, let $M_i$ be the geodesic surface of $X$ consisting of all copies of the facet $i$ of $P$; then we can write $c = c_1[M_1] + \dots + c_{k}[M_{k}]$, for some possibly non-unique coefficients $c_i \in \ZZ_2$.

    We will construct $Y$ by coloring $Q$ as follows. Label its facet isometric to $P$ with $k+1$, and label its adjacent facets $1,\dots,k$ as in the labeling of $P$ (which can be done since $\Fac(P) \simeq \operatorname{adj}(P)$). Then, define a partial coloring $\lambda^* \colon \{1,\dots,k+1\} \to \ZZ_2^m \times \ZZ_2$:
    \begin{equation}
        \lambda^*(i) \coloneqq
        \left\lbrace\,
        \begin{array}{@{}l@{}r@{\,}r@{}r@{\quad}l@{}}
            (&\lambda(i), &c_i&) & \text{if }i\le k, \\
            (&\mathbf{0}, &1&) & \text{if }i = k+1.
        \end{array}
        \right.
    \end{equation}
    As in Remark~\ref{rmk:coloring-extend}, we can extend $\lambda^*$ to a coloring $\til{\lambda}$ of $Q$. We claim that the resulting $4$-manifold $Y \coloneqq M(Q, \til\lambda)$ satisfies the condition.

    First, note that $X$ is embedded in $Y$ as a connected component of the preimage of facet $k+1$. Indeed, such a component is a real toric manifold over $P$; by Proposition~\ref{prop:coloring-sub}, the coloring is given by the images of $\til{\lambda}(i), 1\le i \le k$ in the quotient by $\operatorname{span}(\til{\lambda}(k+1))$, and clearly such a coloring is isomorphic to the original $\lambda$.

    The class $w_1(\nu_Y(X)) \in H^1(X; \ZZ_2) \simeq \Hom(\pi_1(X); \ZZ_2)$ is the monodromy of the double cover of $X$ ($S^0$-bundle) contained in $\nu_Y(X) \setminus X$, and can be computed as follows: given a closed loop $\gamma \in \pi_1(X)$ in general position, intersecting facets $i_1, \dots, i_s$ in order, we have \begin{equation}
        w_1(\nu_Y(X))(\gamma) = [\til{\lambda}(i_1) + \dots + \til{\lambda}(i_s)] \cdot \til{\lambda}(k+1) = c_{i_1} + \dots + c_{i_s}.
    \end{equation}
    On the other hand, we have
    $
        c(\gamma) = \sum_{i=1}^{k} c_i(\gamma \cdot \Sigma_i)
    $. Each intersection of $\gamma$ with facet $i_s$ contributes $c_s$ to the sum, which then equals $w_1(\nu_Y(X))(\gamma)$. Since $\gamma$ was arbitrary, the proof is complete.
\end{proof}

We will now use this proposition to reproduce the proof of Theorem~\ref{thm:main2}.

In the steps up to dimension $4$, the class $c$ of Corollary~\ref{cor:nonzero-whitney} is not guaranteed to be a sum of hypersurfaces, since Proposition~\ref{prop:3in4} does not necessarily give a small cover. Hence, we will skip to dimension $3$ and start from a small cover of the right-angled dodecahedron. Using the computer algebra system SageMath~\cite{sagemath}, we apply formula~(\ref{eq:sw-degree}) to the $25$ small covers in question, discovering that $22$ of these have non-vanishing second Stiefel--Whitney class.\footnote{Referring to~\cite[Table~1]{garrison-scott}, the other three are at rows $10$ and $13$ in the left column, and row $3$ in the right column. They are exactly the small covers of the dodecahedron having an isometry of order $3$.}

Let $X$ be one of these $22$ small covers and let $c_{(3)} \in H^1(X; \ZZ_2)$ be any class such that $w_2(X) \cdot c_{(3)} \ne 0$. Using Proposition~\ref{prop:3in4} with the $120$-cell as $Q$ and $c = c_{(3)}$, we obtain a $4$-manifold $Y$ with $w_3(Y) \ne 0$.

Even if $Y$ is not a small cover, at this point we choose $c = c_{(4)} = w_1(Y)$, which is always a sum of hypersurfaces: for any real toric manifold over a polytope, we have
\begin{equation}
    w_1(M(P, \lambda)) = \sum_{i \in E(P, \lambda)} [M_i]
\end{equation}
where
\begin{equation}
    E(P, \lambda) \coloneqq \{i\in \operatorname{Fac}(P)\mid \text{$\lambda(i)$ has an even number of ones}\}.
\end{equation}
Hence, to finally obtain an orientable, non-\spinc{} $5$-manifold, we just need to choose a suitable $Q$ having the $120$-cell as a facet. However, since compact right-angled hyperbolic $5$-polytopes do not exist, we will construct a right-angled $5$-manifold based on a $5$-polytope $Q_0$ discovered by Bugaenko~\cite{bugaenko}, with Coxeter diagram
\begin{equation}
      \begin{tikzpicture}
            \node[dot, fill=black] (a) at (0,0) {};
            \node[dot, fill=black] (b) at (1,0) {};
            \node[dot, fill=black] (c) at (2,0) {};
            \node[dot, fill=black] (d) at (3,0) {};
            \node[dot, fill=black] (e) at (4,0) {};
            \node[dot, fill=black] (f) at (5,0) {};
            \node[dot, fill=black] (g) at (6,0) {};
        	\draw[draw=black, double distance=3pt, thin, solid] (a) -- (b);
        	\draw[draw=black, thin, solid] (a) -- (b);
        	\draw[draw=black, thin, solid] (b) -- (c);
        	\draw[draw=black, thin, solid] (c) -- (d);
        	\draw[draw=black, thin, solid] (d) -- (e);
        	\draw[draw=black, double distance=2.5pt, thin, solid] (e) -- (f);
        	\draw[draw=black, thin, dashed] (f) -- (g);
        \end{tikzpicture}
\end{equation}
This polytope is, perhaps unsurprisingly, also arithmetic of simplest type and defined over $\QQ(\sqrt{5})$. It is a compact hyperbolic prism over a $4$-simplex, with the two bases isometric to the fundamental simplices $\Delta_3, \Delta_4$ of the order-$3$ $120$-cell and right-angled $120$-cell respectively.

Let $Z$ be a closed hyperbolic $4$-manifold tessellated by the order-$3$ hyperbolic $120$-cell; examples include~\cite{conder-maclachlan, c-long}. Up to passing to a finite-index cover, by an injectivity radius argument, we may assume that each $120$-cell is embedded in $Z$ and has pairwise distinct neighboring cells. We can then use the tessellation of $Z$ into copies of $\Delta_3$ to glue copies of $Q_0$ along both sides of $Z$ (in a local sense). This can be done in such a way as to obtain a $5$-manifold with right-angled corners $W$, homeomorphic to the determinant line bundle on $Z$ and hence orientable. The boundary of the manifold $W$ consists of many embedded right-angled $120$-cells, each having pairwise distinct adjacent facets. Hence, it can be used for our construction, ultimately yielding an orientable, non-\spinc{} hyperbolic $5$-manifold, tessellated by the polytope $Q_0$.


We note that $W$ is likely to have high combinatorial complexity, and so is \emph{a fortiori} the final non-\spinc{} $5$-manifold.
\setlength\bibitemsep{1ex}
\printbibliography[heading=bibintoc, title={References}]
\end{document}